\documentclass[12pt]{amsart}
\usepackage{amssymb,latexsym}

\numberwithin{equation}{section}

\newtheorem{theorem}{Theorem}[section]
\newtheorem{proposition}[theorem]{Proposition}
\newtheorem{conjecture}[theorem]{Conjecture}

\theoremstyle{definition}

\newtheorem{remark}[theorem]{Remark}

\def\cc{\mathbf{c}}
\def\ee{\mathbf{e}}
\def\gg{\mathbf{g}}
\def\xx{\mathbf{x}}
\def\yy{\mathbf{y}}
\def\TT{\mathbb{T}}
\def\PP{\mathbb{P}}
\def\ZZ{\mathbb{Z}}
\def\QQ{\mathbb{Q}}

\def\Fcal{\mathcal{F}}

\newcommand{\overunder}[2]{
\!\begin{array}{c}
\scriptstyle{#1}\\[-.1in]
-\!\!\!-\!\!\!-\\[-.1in]
\scriptstyle{#2}
\end{array}
\!
}

\begin{document}

\title{On tropical dualities in cluster algebras}

\author{Tomoki Nakanishi}
\address{\noindent Graduate School of Mathematics, Nagoya University, Nagoya,
464-8604, Japan}
\email{nakanisi@math.nagoya-u.ac.jp}

\author{Andrei Zelevinsky}
\address{\noindent Department of Mathematics, Northeastern University,
Boston, MA 02115, USA}
\email{andrei@neu.edu}

\begin{abstract}
We study two families of integer vectors playing a crucial part in the structural theory of cluster algebras:
the $\gg$-vectors parameterizing cluster variables, and the $\cc$-vectors parameterizing the coefficients.
We prove two identities relating these vectors to each other.
The proofs depend on the sign-coherence assumption for $\cc$-vectors that still remains unproved in full generality.
  \end{abstract}

\date{January 19, 2011; revised April 13, 2011 and May 24, 2011}

\thanks{Research of A.~Z. supported by the NSF grant DMS-0801187.}

\maketitle

\section{Introduction}

In the theory of cluster algebras, the central role is played by two families of elements obtained by certain algebraic recurrences:
\emph{cluster variables} and \emph{coefficients} (V.~Fock and A.~Goncharov refer to them as
\emph{$a$-coordinates}, and \emph{$x$-coordinates}, respectively).
In this paper we study certain \emph{tropical limits} of these elements: two families of integer vectors (\emph{$\gg$-vectors} and
\emph{$\cc$-vectors}, respectively) introduced in \cite{ca4}.
Despite the fact that these two families are given by very different piecewise-polynomial recurrences, it turns our that they are closely related to each other.
We refer to the identities involving $\gg$-vectors and $\cc$-vectors as \emph{tropical dualities}.
The main result of this paper is Theorem~\ref{th:main} containing two such dualities.

We start by briefly recalling the formalism of cluster algebras (more details can be found in \cite{ca4}).
In the heart of this formalism there are several discrete dynamical systems on a $n$-regular tree
given by birational recurrences and their ``tropical'' versions.
More precisely, we fix a positive integer~$n$, and denote by $\TT_n$ an \emph{$n$-regular tree}
whose edges are labeled by the numbers $1, \dots, n$,
so that the $n$ edges emanating from each vertex receive different labels.
We write $t \overunder{k}{} t'$ to indicate that vertices
$t,t'\in\TT_n$ are joined by an edge labeled by~$k$.

\smallskip

We need a little algebraic preparation.
We call a \emph{semifield} an
abelian multiplicative group~$\PP$, supplied with the addition operation
$\otimes$, which is commutative, associative, and distributive with respect to multiplication.
Let $\ZZ \PP$ denote the integer group ring of (the multiplicative group of) $\PP$
(note that its definition\emph{ignores} the addition $\oplus$).
It is easy to show that (the multiplicative group of) $\PP$ is torsion-free (see the remark before Definition~5.3 in \cite{ca1}), hence $\ZZ \PP$ is a domain, hence it has the  field of fractions $\QQ (\PP)$.

We call an \emph{ambient field} a field
$\Fcal$ isomorphic to the field of rational functions in $n$ indeterminates
with the coefficients from $\QQ(\PP)$.

\smallskip

Fix a semifield $\PP$ and an ambient field $\Fcal$.
Following  \cite{ca4}, we call a \emph{(labeled) seed}
a triple $(\xx, \yy, B)$, where
\begin{itemize}
\item
$B = (b_{ij})$ is an $n\!\times\! n$ integer matrix
which is \emph{skew-symmetrizable}, that~is, its transpose
$B^T$ is equal to $- D B D^{-1}$ for some diagonal matrix $D$ with positive integer diagonal entries
$d_1, \dots, d_n$;
\item
$\yy = (y_1, \dots, y_n)$ is an $n$-tuple
of elements of~$\PP$, and
\item
$\xx = (x_1, \dots, x_n)$ is an $n$-tuple of elements of~$\Fcal$ forming a \emph{free generating set},
that is, being algebraically independent
and such that $\Fcal = \QQ(\PP)(x_1, \dots, x_n)$.
\end{itemize}

We refer to $B$ as the \emph{exchange matrix} of a seed, $\yy$ as the
\emph{coefficient tuple}, and $\xx$ as the \emph{cluster}.

\smallskip

Throughout the paper we use the notation $[b]_+ = \max(b,0)$.
For $k = 1, \dots, n$, the \emph{seed mutation} $\mu_k$ transforms
$(\xx, \yy, B)$ into the labeled seed
$\mu_k(\xx, \yy, B)=(\xx', \yy', B')$ defined as follows:
\begin{itemize}
\item
The entries of the exchange matrix $B'=(b'_{ij})$ are given by
\begin{equation}
\label{eq:matrix-mutation}
b'_{ij} =
\begin{cases}
-b_{ij} & \text{if $i=k$ or $j=k$;} \\[.05in]
b_{ij} + [b_{ik}]_+ [b_{kj}]_+ - [-b_{ik}]_+ [-b_{kj}]_+
 & \text{otherwise.}
\end{cases}
\end{equation}
\item
The coefficient tuple $\yy'=(y_1',\dots,y_n')$ is given by
\begin{equation}
\label{eq:y-mutation}
y'_j =
\begin{cases}
y_k^{-1} & \text{if $j = k$};\\[.05in]
y_j y_k^{[b_{kj}]_+}
(y_k \oplus 1)^{- b_{kj}} & \text{if $j \neq k$}.
\end{cases}
\end{equation}
\item
The cluster $\xx'=(x_1',\dots,x_n')$ is given by
$x_j'=x_j$ for $j\neq k$,
whereas $x'_k \in \Fcal$ is determined
by the \emph{exchange relation}
\begin{equation}
\label{eq:exchange-rel-xx}
x'_k = \frac
{y_k \ \prod x_i^{[b_{ik}]_+}
+ \ \prod x_i^{[-b_{ik}]_+}}{(y_k \oplus 1) x_k} \, .
\end{equation}
\end{itemize}

It is easy to see that $B'$ is
skew-symmetrizable (with the same choice of $D$),
implying that $(\xx', \yy', B')$ is indeed a seed.
Furthermore, the seed mutation $\mu_k$ is involutive,
that is, it transforms $(\xx', \yy', B')$ into the original seed
$(\xx, \yy, B)$.
We shall also use the notation $B' = \mu_k(B)$ (resp.~$(\yy', B') = \mu_k(\yy, B)$) and call the transformation
$B \mapsto B'$ the \emph{matrix mutation} (resp.~$(\yy, B) \mapsto (\yy', B')$ the \emph{$Y$-seed mutation}).

\smallskip

A \emph{seed pattern} is an assignment
of a seed $(\xx_t, \yy_t, B_t)$
to every vertex $t \in \TT_n$, such that the seeds assigned to the
endpoints of any edge $t \overunder{k}{} t'$ are obtained from each
other by the seed mutation~$\mu_k$.
We write:
\begin{equation}
\label{eq:seed-labeling}
\xx_t = (x_{1;t}\,,\dots,x_{n;t})\,,\quad
\yy_t = (y_{1;t}\,,\dots,y_{n;t})\,,\quad
B_t = (b_{ij;t})\,.
\end{equation}
We also refer to a family $(\yy_t, B_t)|_{t \in \TT_n}$ as a \emph{$Y$-seed pattern}, and
a family $(B_t)|_{t \in \TT_n}$ as an \emph{exchange matrix pattern}.

It is convenient to fix a vertex (\emph{root}) $t_0 \in \TT_n$.
Then a seed pattern is uniquely determined by a seed at $t_0$, which can be chosen arbitrarily.
We will just write this initial seed as $(\xx, \yy, B)$, with
\begin{equation}
\label{eq:initial-seed-labeling}
\xx = (x_{1}\,,\dots,x_{n})\,,\quad
\yy = (y_{1}\,,\dots,y_{n})\,,\quad
B = (b_{ij})\,.
\end{equation}

\smallskip

The \emph{cluster algebra} associated with a seed $(\xx, \yy, B)$ (or rather with the corresponding seed pattern) is the
$\ZZ \PP$-subalgebra of the ambient field $\Fcal$ generated by all \emph{cluster variables} $x_{j;t}$.
One of the central problems in the theory of cluster algebras is to find explicit expressions for
all $x_{j;t}$ and $y_{j;t}$ as rational functions in $x_1, \dots, x_n, y_1, \dots, y_n$.
A step in this direction was made in \cite{ca4}, where the following was proved.

\begin{proposition}
\label{pr:F-g-c}
{\rm \cite[Proposition~3.13, Corollary~6.3]{ca4}}
Every pair $(B;t_0)$ gives rise to a family of polynomials
$F_{j;t} = F_{j;t}^{B;t_0} \in \ZZ[u_1, \dots, u_n]$ and two families of integer vectors
$\cc_{j;t} = \cc_{j;t}^{B;t_0} = (c_{1j;t}, \dots, c_{nj;t}) \in \ZZ^n$ and
$\gg_{j;t} = \gg_{j;t}^{B;t_0} = (g_{1j;t}, \dots, g_{nj;t}) \in \ZZ^n$
(where $j \in \{1, \dots, \}$ and $t \in \TT_n$)
with the following properties:
\begin{enumerate}
\item
Each $F_{j;t}$ is not divisible by any $u_i$, and can be expressed as a ratio
of two polynomials in $u_1, \dots, u_n$ with positive integer coefficients, thus can be evaluated
in every  semifield~$\PP$.
\item
For any $j$ and $t$, we have
\begin{equation}
\label{eq:yjt}
y_{j;t} = y_1^{c_{1j;t}} \cdots y_n^{c_{nj;t}} \, \prod_i F_{i;t}|_\PP (y_1, \dots, y_n)^{b_{ij;t}}\ .
\end{equation}
\item
For any $j$ and $t$, we have
\begin{equation}
\label{eq:xjt}
x_{j;t} = x_1^{g_{1j;t}} \cdots x_n^{g_{nj;t}} \,\frac{F_{j;t}|_\Fcal(\hat y_1, \dots, \hat y_n)}
{F_{j;t}|_\PP (y_1, \dots, y_n)} \ ,
\end{equation}
where the elements $\hat y_j$ are given by
$$\hat y_j = y_j \prod_i x_i^{b_{ij}} \ .$$
\end{enumerate}
\end{proposition}

Following \cite{ca4}, we refer to $\gg_{j;t}$ as the \emph{$\gg$-vector} of a cluster variable $x_{j;t}$.
It is instructive to view the $\gg$-vectors as some kind of discrete (or \emph{tropical}) limits of cluster variables.
In fact, as explained in \cite[Remark~7.15]{ca4}, these vectors are expected to provide a parametrization of cluster variables (and more generally, cluster monomials)
by ``Langlands dual tropical $Y$-seed patterns" in accordance with a conjecture by V.~Fock and A.~Goncharov \cite[Conjecture~5.1]{fg}.
Some properties of $\gg$-vectors and $F$-polynomials were established in \cite{ca4} but several basic properties still remain conjectural.

\smallskip

Comparing \eqref{eq:yjt} and \eqref{eq:xjt}, it is natural to view the \emph{$\cc$-vector} $\cc_{j;t}$ as a
tropical limit of a coefficient $y_{j;t}$.
These vectors also made their appearance in \cite{ca4}, although not under this name, and only as a tool for studying
$\gg$-vectors and $F$-polynomials.
In this note we focus on the properties of $\cc$-vectors and their relationships with $\gg$-vectors.

\smallskip

The following (unfortunately, still conjectural) \emph{sign-coherence} property of $\cc$-vectors is crucial for our analysis:
\begin{equation}
\label{eq:C-sign-coherence}
\text{Each vector $\cc_{j;t}$ has either all entries nonnegative or all entries nonpositive.}
\end{equation}
As shown in \cite[Proposition~5.6]{ca4}, \eqref{eq:C-sign-coherence} is equivalent to the following conjecture
made in \cite[Conjecture~5.4]{ca4}:
\begin{equation}
\label{eq:F-constant-term}
\text{Each polynomial $F_{j;t}(u_1, \dots, u_n)$ has constant term~$1$.}
\end{equation}
This conjecture (hence the property \eqref{eq:C-sign-coherence})
was proved in \cite{dwz2} for the case of \emph{skew-symmetric} exchange matrices, using
\emph{quivers with potentials} and their representations
(two different proofs were recently given in \cite{nag,pla}).

\smallskip

We denote by $C_t^{B;t_0}$ (resp.~$G_t^{B;t_0}$) the integer matrix with columns  $\cc_{1;t}, \dots, \cc_{n;t}$
(resp. with columns  $\gg_{1;t}, \dots, \gg_{n;t}$), where $B$ is the exchange matrix at $t_0$.
In particular, we have
\begin{equation}
\label{eq:C-G-initial}
C_{t_0}^{B;t_0} = G_{t_0}^{B;t_0} = I \quad \text{(the identity matrix).}
\end{equation}

The main results of this paper are the following two identities.

\begin{theorem}
\label{th:main}
Under the assumption \eqref{eq:C-sign-coherence},
for any skew-symmetrizable exchange matrix~$B$, and any $t_0, t \in \TT_n$,
we have
\begin{equation}
\label{eq:G-C-inverse}
(G_{t}^{B;t_0})^T = (C_t^{-B^T;t_0})^{-1},
\end{equation}
and
\begin{equation}
\label{eq:C-C-opposite}
C_t^{B;t_0} = (C_{t_0}^{-B_t;t})^{-1},
\end{equation}
where $t \mapsto B_t$ is the exchange matrix pattern on $\TT_n$ such that
$B_{t_0} = B$, and $B^T$ stands for the transpose matrix of~$B$.
\end{theorem}

Note that in the case when $B$ is skew-symmetric, \eqref{eq:G-C-inverse}
takes the form $(G_{t}^{B;t_0})^T = (C_t^{B;t_0})^{-1}$.
In this case it was proved in \cite[Proposition~4.1]{nak}.

One can show that identities \eqref{eq:G-C-inverse} and \eqref{eq:C-C-opposite}
imply most of the conjectures made in \cite{ca4} and recast in \cite{dwz2}, namely
Conjectures~1.1 - 1.4 and 1.6 from \cite{dwz2}.
For instance, by combining them, we obtain the equality
\begin{equation}
\label{eq:G-C-opposite}
(G_t^{B;t_0})^T = C_{t_0}^{B_t^T;t},
\end{equation}
which can be shown to imply \cite[Conjecture~1.6]{dwz2}.
More details will be given in Section~\ref{sec:corollaries}.

\smallskip

In contrast with \cite{dwz2}, the proofs of \eqref{eq:G-C-inverse} and \eqref{eq:C-C-opposite}
given below do not use any categorical interpretation of $\gg$- or $\cc$-vectors but
are based on the analysis of recurrences satisfied by them.
To describe these recurrences we need to develop some matrix formalism.

\smallskip

We extend the notation $[b]_+ = \max(0,b)$ to matrices, writing $[B]_+$ for the matrix obtained from $B$ by applying
the operation $b \mapsto [b]_+$ to all entries of $B$.
For a matrix index~$k$, we denote by $B^{\bullet k}$ the matrix obtained from $B$ by replacing all entries
outside of the $k$-th column with zeros; the matrix $B^{k \bullet}$
is defined similarly using the $k$-th row instead of the column.
Note that the operations $B \mapsto [B]_+$ and $B \mapsto B^{\bullet k}$ commute with each other, making
the notation $[B]_+^{\bullet k}$ (and $[B]_+^{k \bullet}$) unambiguous.

\smallskip

Using this formalism, we can rephrase the sign-coherence condition \eqref{eq:C-sign-coherence} for a matrix
$C = C_t^{B;t_0}$ as follows:
\begin{equation}
\label{eq:C-sign-coherence-matrix-form}
\text{For every~$j$, there exists the sign $\varepsilon_j(C) = \pm 1$ such that
$[-\varepsilon_j(C) C]_+^{\bullet j} = 0$.}
 \end{equation}

Let $J_k$ denote the diagonal matrix obtained from the identity matrix by replacing the $(k,k)$-entry with $-1$.
We deduce \eqref{eq:G-C-inverse} from the following proposition.

\begin{proposition}
\label{pr:G-C-mutation-right-end}
Suppose $t \overunder{\ell}{} t'$ in $\TT_n$, and let
$C = C_t^{B;t_0}$, $G = G_t^{B;t_0}$, $C' = C_{t'}^{B;t_0}$, and
$G' = G_{t'}^{B;t_0}$.
Then, under the assumption that $C$ satisfies \eqref{eq:C-sign-coherence-matrix-form}, we have
\begin{equation}
\label{eq:C'G'-thru-CG-right-end}
C' = C (J_\ell + [\varepsilon_\ell(C) B_t]_+^{\ell \bullet}),
\quad G' = G (J_\ell + [- \varepsilon_\ell(C) B_t]_+^{\bullet \ell}) \ .
\end{equation}
\end{proposition}

As for \eqref{eq:C-C-opposite}, we prove it together with the following proposition
that provides a recurrence relation for the matrices $C_t^{B;t_0}$ ``at the opposite end."
We still assume that every matrix of the form $C = C_t^{B;t_0}$ satisfies \eqref{eq:C-sign-coherence}
or equivalently \eqref{eq:C-sign-coherence-matrix-form}.

\begin{proposition}
\label{pr:C-mutation-left-end}
Suppose $t_0 \overunder{k}{} t_1$ in $\TT_n$, and let $B_1 = \mu_k(B)$.
For $t \in \TT_n$, abbreviate $C_t = C_t^{B;t_0}$,  and
$C'_t = C_{t}^{B_1;t_1}$.
Then  we have
\begin{equation}
\label{eq:C'-thru-C-left-end}
C'_t = (J_k + [-\varepsilon_k(C_{t_0}^{-B_t;t}) B]_+^{k \bullet}) C_t  \ .
\end{equation}
\end{proposition}

\smallskip

Proposition~\ref{pr:G-C-mutation-right-end} and the identity~\eqref{eq:G-C-inverse}
will be proved in Section~\ref{sec:G-C-inverse-proofs}, while
Proposition~\ref{pr:C-mutation-left-end} and \eqref{eq:C-C-opposite}
will be proved in Section~\ref{sec:C-C-inverse-proofs}.
In the last section we give some corollaries.

\section{Proofs of Proposition~\ref{pr:G-C-mutation-right-end} and identity~\eqref{eq:G-C-inverse}}
\label{sec:G-C-inverse-proofs}

\begin{proof}[Proof of Proposition~\ref{pr:G-C-mutation-right-end}]
We start with the following identity involving the function $[b]_+ = \max(b,0)$:
for any real numbers $b$ and $c$, and a sign $\varepsilon = \pm 1$, we have
\begin{equation}
\label{eq:c+b+identity}
[c]_+ [b]_+ - [-c]_+ [-b]_+ = c [\varepsilon b]_+ + b [-\varepsilon c]_+ \ .
\end{equation}
This is a consequence of an obvious identity
\begin{equation}
\label{eq:b+identity}
[b]_+ - [-b]_+ = b \ .
\end{equation}
Indeed, we have
$$c [b]_+ + b [- c]_+ =
([c]_+ - [-c]_+) [b]_+ + ([b]_+ - [-b]_+) [-c]_+ =
[c]_+ [b]_+ - [-c]_+ [-b]_+ \ ,$$
proving \eqref{eq:c+b+identity} for $\varepsilon = 1$; the case $\varepsilon = -1$
is proved similarly.

The first equality in \eqref{eq:C'G'-thru-CG-right-end} was essentially shown in \cite[(3.2)]{ca3}
(in a different notation).
For the convenience of the reader, we reproduce the argument here.
Recall from  \cite[Remarks~3.2, 3.14]{ca4} that the relationship between
$C = C_t^{B;t_0}$ and $C' = C_{t'}^{B;t_0}$ can be described as follows:
if $\tilde B_t$ is a $2n \times n$ integer matrix with the top $n \times n$ block $B_t$ and the bottom block
$C_t^{B;t_0}$, then $\tilde B_{t'}$ is obtained from $\tilde B_t$ by
the matrix mutation $\mu_\ell$ given by the same formula as in \eqref{eq:matrix-mutation}.
Thus, we get
\begin{equation}
\label{eq:C'-thru-C-matrix-mutation}
c'_{ij} =
\begin{cases}
-c_{ij} & \text{if $j=\ell$;} \\[.05in]
c_{ij} + [c_{i \ell}]_+ [b_{\ell j;t}]_+ - [-c_{i \ell}]_+ [-b_{\ell j;t}]_+
 & \text{otherwise.}
\end{cases}
\end{equation}
Using \eqref{eq:c+b+identity} and rewriting the resulting formula in the matrix form, we get
\begin{equation}
\label{eq:C'-thru-C-matrix-form-general}
C' = C (J_\ell + [\varepsilon B_t]_+^{\ell \bullet})  +  [- \varepsilon C]_+^{\bullet \ell} B_t
\end{equation}
for any choice of the sign $\varepsilon$.
Setting $\varepsilon = \varepsilon_\ell(C)$ and remembering \eqref{eq:C-sign-coherence-matrix-form}, we see
that the second term in \eqref{eq:C'-thru-C-matrix-form-general} disappears, implying the desired equality.

To prove the second equality in \eqref{eq:C'G'-thru-CG-right-end},
we rewrite \cite[(6.12)-(6.13)]{ca4} in the matrix form as follows:
\begin{equation}
\label{eq:G'-thru-G-matrix-form-general}
G' = G (J_\ell + [- \varepsilon B_t]_+^{\bullet \ell})   - B [- \varepsilon C]_+^{\bullet \ell}
\end{equation}
(again for any choice of the sign $\varepsilon$).
Again setting $\varepsilon = \varepsilon_\ell(C)$ makes the second term in \eqref{eq:G'-thru-G-matrix-form-general}
disappear, completing the proof of Proposition~\ref{pr:G-C-mutation-right-end}.
\end{proof}

\begin{proof}[Proof of identity~\eqref{eq:G-C-inverse}]
We prove \eqref{eq:G-C-inverse} by induction on the distance between $t$ and $t_0$ in the tree $\TT_n$.
For $t=t_0$, the claim follows from \eqref{eq:C-G-initial}.
It remains to show that if \eqref{eq:G-C-inverse} holds for some $t \in \TT_n$ then it also holds
for $t' \in \TT_n$ such that $t \overunder{\ell}{} t'$.
This implication is a consequence of \eqref{eq:C'G'-thru-CG-right-end} combined with the following
three simple observations:
\begin{itemize}
\item
For every integer square matrix~$B$ and an index~$\ell$ such that
$b_{\ell \ell} = 0$, we have
\begin{equation}
\label{eq:J+B-ell-transpose-inverse}
(J_\ell + B^{\bullet \ell})^T =  J_\ell + {(B^T)}^{\ell \bullet}, \quad
(J_\ell + B^{\ell \bullet})^{-1} = J_\ell + B^{\ell \bullet} \ .
\end{equation}
\item
The replacement of the initial exchange matrix $B_{t_0} = B$ with
$- B^T$ leads to the replacement of each $B_t$ with $- B_t^T$, and we have
\begin{equation}
\label{eq:replacing-B-by--BT}
C_t^{-B^T;t_0} = D C_t^{B;t_0} D^{-1},
\end{equation}
where $D$ is a diagonal $n \times n$ matrix with positive integer entries such that
$- B^T = D B D^{-1}$.
\item
In particular, we have $\varepsilon_\ell(C_t^{-B^T;t_0}) = \varepsilon_\ell(C_t^{B;t_0})$
for every matrix index~$\ell$.
\end{itemize}
\end{proof}

\begin{remark}
As in Proposition~\ref{pr:G-C-mutation-right-end}, let us abbreviate
$C = C_t^{B;t_0}$, and $G = G_t^{B;t_0}$.
In the  equality \eqref{eq:G'-thru-G-matrix-form-general},
the fact that two choices of the sign give the same answer implies the following identity
for $G$ (shown already in \cite[(6.14)]{ca4}):
\begin{equation}
\label{eq:GBt=BC}
G B_t  = B C \ .
\end{equation}
Combining this identity with \eqref{eq:G-C-inverse} and \eqref{eq:replacing-B-by--BT}, we obtain
the following:
\begin{equation}
\label{eq:Bt-thru-B-and-C}
D B_t  = C^T \cdot DB \cdot C \ ,
\end{equation}
where the diagonal matrix~$D$ is as in \eqref{eq:replacing-B-by--BT}.
This shows that the exchange matrix~$B_t$ at every vertex $t \in \TT_n$ (that is, the top part of the $2n \times n$
matrix~$\tilde B_t$) is determined by the initial exchange matrix~$B$ and the matrix~$C$ which is the bottom part
of~$\tilde B_t$.
\end{remark}

\section{Proofs of Proposition~\ref{pr:C-mutation-left-end} and identity~\eqref{eq:C-C-opposite}}
\label{sec:C-C-inverse-proofs}

We prove \eqref{eq:C'-thru-C-left-end} and \eqref{eq:C-C-opposite} by a simultaneous induction.
To prepare the ground for it, we note that, under the notation and assumptions of
Proposition~\ref{pr:C-mutation-left-end}, we have
$B_1^{k \bullet} = - B^{k \bullet}$, and
$(C_{t_1}^{-B_t;t})^{\bullet k} = - (C_{t_0}^{-B_t;t})^{\bullet k}$, hence
$\varepsilon_k(C_{t_1}^{-B_t;t}) = -\varepsilon_k(C_{t_0}^{-B_t;t})$.
In view of the second equality in \eqref{eq:J+B-ell-transpose-inverse}, this allows us to interchange
$t_0$ and $t_1$ in \eqref{eq:C'-thru-C-left-end}.
Thus, it suffices to prove \eqref{eq:C'-thru-C-left-end} under the assumption that
\begin{equation}
\label{eq:t0-between-t-t1}
\text{$t_0$ belongs to the (unique) path between $t$ and $t_1$ in $\TT_n$.}
\end{equation}
In other words, \eqref{eq:t0-between-t-t1} means that $d(t,t_1) = d(t,t_0) + 1$, where
$d(t,t_0)$ denotes the distance between $t$ and $t_0$ in $\TT_n$.

For $d \geq 0$, we denote by $(I_d)$ and $(II_d)$ the following two statements:
\begin{enumerate}
\item[$(I_d)$]
The equality \eqref{eq:C-C-opposite} holds for $d(t,t_0) = d$.
\smallskip
\item[$(II_d)$]
The equality \eqref{eq:C'-thru-C-left-end} holds
under the assumption \eqref{eq:t0-between-t-t1} whenever $d(t,t_0) = d$.
\end{enumerate}
We will prove both $(I_d)$ and $(II_d)$ simultaneously by induction on~$d$.

For $d=0$ we have $t = t_0$, and both \eqref{eq:C-C-opposite} and \eqref{eq:C'-thru-C-left-end}
are immediate from the definitions.

It remains to prove the implications $(I_d) \ \& \ (II_d) \Longrightarrow (I_{d+1})$ and
$(I_d) \ \& (II_d)\ \Longrightarrow (II_{d+1})$.

\begin{proof}[Proof of the implication $(I_d) \ \& \ (II_d) \Longrightarrow (I_{d+1})$]
We need to show that if \eqref{eq:C-C-opposite} and \eqref{eq:C'-thru-C-left-end}
hold  for some $t, t_0$, and $t_1$ satisfying \eqref{eq:t0-between-t-t1}, then
\eqref{eq:C-C-opposite} also holds if $t_0$ is replaced with $t_1$,
and $B$ is replaced by $B_1 = \mu_k(B)$
(note that this replacement leaves $B_t$ unchanged by the definition of an exchange matrix pattern).
Let us abbreviate $C = C_t^{B;t_0}$.
To prove the desired equality
$$C_t^{B_1;t_1} = (C_{t_1}^{-B_t;t})^{-1}$$
we use \eqref{eq:C'-thru-C-left-end}, the second equality in \eqref{eq:J+B-ell-transpose-inverse},
the inductive assumption $C^{-1} = C_{t_0}^{-B_t;t}$, and the first equality in
\eqref{eq:C'G'-thru-CG-right-end} to get
\begin{align*}
(C_t^{B_1;t_1})^{-1} &= ((J_k + [-\varepsilon_k(C_{t_0}^{-B_t;t}) B]_+^{k \bullet})\  C)^{-1}\\
&= C^{-1} (J_k + [-\varepsilon_k(C_{t_0}^{-B_t;t}) B]_+^{k \bullet})\\
&= C_{t_0}^{-B_t;t} (J_k + [\varepsilon_k(C_{t_0}^{-B_t;t}) \cdot (-B)]_+^{k \bullet})\\
& = C_{t_1}^{-B_t;t}\
\end{align*}
(to see that the last equality is indeed an instance of \eqref{eq:C'G'-thru-CG-right-end},
note that the matrix mutation commutes with the operation $B \mapsto -B$, thus the exchange matrix pattern
that assigns $-B_t$ to $t$, also assigns $-B$ to $t_0$).  \end{proof}


\begin{proof}[Proof of the implication $(I_d) \ \& \ (II_d) \Longrightarrow (II_{d+1})$]
This proof is more involved than the previous one.
Let us first summarize our assumptions, and what has to be proven.

Suppose $d(t_0, t) = d$ in $\TT_n$, and let $t_1, t' \in \TT_n$ be such that $t_0 \overunder{k}{} t_1$,
$t \overunder{\ell}{} t'$, and $t$ and $t_0$ belong to the unique path between $t'$ and $t_1$.
Thus we have $d(t_1, t) = d(t_0, t') = d+1$, and $d(t_1, t') = d+2$.
The assumption $(I_d)$ is just the equality \eqref{eq:C-C-opposite}.
Since we have already shown that $(I_d)$ and $(II_d)$ imply $(I_{d+1})$, we can also assume the equalities
\begin{equation}
\label{eq:C-C-opposite-distance-d+1}
C_t^{B_1;t_1} = (C_{t_1}^{-B_t;t})^{-1}, \quad
C_{t'}^{B;t_0} = (C_{t_0}^{-B_{t'};t'})^{-1} \ .
\end{equation}
Now the assumption $(II_d)$ gives us the equalities
\begin{align}
\label{eq:C'-thru-C-left-end-d}
&C_{t}^{B_1;t_1} = (J_k + [-\varepsilon_k(C_{t_0}^{-B_t;t}) B]_+^{k \bullet}) C_t^{B;t_0},\\
\nonumber
&C_{t_0}^{-B_{t'};t'} = (J_\ell + [\varepsilon_{\ell}(C_{t}^{B;t_0}) B_t]_+^{\ell \bullet}) C_{t_0}^{-B_t;t},
\end{align}
while our goal is to prove that
\begin{equation}
\label{eq:C'-thru-C-left-end-d+1}
C_{t'}^{B_1;t_1} = (J_k + [-\varepsilon_k(C_{t_0}^{-B_{t'};t'}) B]_+^{k \bullet}) C_{t'}^{B;t_0}\ .
\end{equation}

To prove \eqref{eq:C'-thru-C-left-end-d+1}, we invoke the following
equalities which are instances of \eqref{eq:C'G'-thru-CG-right-end}:
\begin{equation}
\label{eq:C'-thru-C-right-end-two-times}
C_{t'}^{B;t_0} = C_{t}^{B;t_0} (J_\ell + [\varepsilon_\ell(C_{t}^{B;t_0}) B_t]_+^{\ell \bullet}), \quad
C_{t'}^{B_1;t_1} = C_{t}^{B_1;t_1} (J_\ell + [\varepsilon_\ell(C_{t}^{B_1;t_1}) B_t]_+^{\ell \bullet}) \ .
\end{equation}

The following observation is immediate from the first equality in \eqref{eq:C'-thru-C-left-end-d}:
\begin{equation}
\label{eq:only-kth-row}
\text{The transformation
$C_t^{B;t_0} \mapsto C_{t}^{B_1;t_1}$ affects only the entries in the $k$-th row.}
\end{equation}

Now we consider separately the following two mutually exclusive cases:

\smallskip

{\bf Case~1:} the matrix $C_t^{B;t_0}$ has a non-zero entry $(i, \ell)$ for some $i \neq k$.

\smallskip

{\bf Case~2:} the only non-zero entry in the $\ell$-th column of $C_t^{B;t_0}$
is the $(k, \ell)$ entry; equivalently, the $\ell$th column of~$C_t^{B;t_0}$ is of the form
$\varepsilon \ee_k$, where $\ee_1, \dots, \ee_n$ is the standard basis of $\ZZ^n$.
Note that in this case the coefficient $\varepsilon$ is equal to $\pm 1$, since both
$C_t^{B;t_0}$ and its inverse $C_{t_0}^{-B_t;t}$ are integer matrices.

\smallskip

First we deal with Case~1.
In view of \eqref{eq:only-kth-row}, we have $\varepsilon_\ell(C_{t}^{B_1;t_1}) = \varepsilon_\ell(C_{t}^{B;t_0})$
(here we use the assumption \eqref{eq:C-sign-coherence-matrix-form}!).
Combining \eqref{eq:C'-thru-C-right-end-two-times} and \eqref{eq:C'-thru-C-left-end-d}, we see that the desired
equality \eqref{eq:C'-thru-C-left-end-d+1} follows from the equality
\begin{equation}
\label{eq:two-epsilons-case1}
\varepsilon_k(C_{t_0}^{-B_{t'};t'}) =  \varepsilon_k(C_{t_0}^{-B_t;t}) \ .
\end{equation}
Applying the same argument as above to the transformation $C_{t_0}^{-B_t;t} \mapsto C_{t_0}^{-B_{t'};t'}$
given by the second equality in \eqref{eq:C'-thru-C-left-end-d}, we see that it is enough to show that
$C_{t_0}^{-B_t;t}$ satisfies the analog of Case~1, namely has a non-zero entry $(i, k)$ for some $i \neq \ell$.
Now recall that by our assumption $(I_d)$, the matrices $C_t^{B;t_0}$ and $C_{t_0}^{-B_t;t}$ are inverses of each other.
Thus, it is enough to show the following statement from linear algebra:
\begin{eqnarray}
\label{eq:inverses-columns-agree}
&\text{For any indices $k$ and $\ell$, an invertible matrix~$C$ has a non-zero entry}\\
\nonumber
&\text{$(i, \ell)$ for some $i \neq k$
if and only if $C^{-1}$ has a non-zero entry $(i, k)$ for some $i \neq \ell$.}
\end{eqnarray}
The easiest way to prove \eqref{eq:inverses-columns-agree} is to observe
that it becomes obvious after replacing the equivalent statements in question by their negations:
\begin{eqnarray}
\label{eq:inverses-columns-agree-one-entry}
&\text{For any $k$ and $\ell$, the $\ell$th column of~$C$ is of the form
$\varepsilon  \ee_k$}\\
\nonumber
&\text{with $\varepsilon = \pm 1$ if and only if the $k$th column of~$C^{-1}$ is
$\varepsilon  \ee_\ell$.}
\end{eqnarray}

This completes the proof of \eqref{eq:C'-thru-C-left-end-d+1} in Case~1.

\smallskip

Now assume that we are in Case~2, that is, in view of \eqref{eq:inverses-columns-agree-one-entry},
the $\ell$th column of~$C_t^{B;t_0}$ is $\varepsilon \ee_k$, while the
$k$th column of~$(C_t^{B;t_0})^{-1} = C_{t_0}^{-B_t;t}$ is $\varepsilon \ee_\ell$ for some $\varepsilon = \pm 1$.
In particular, we have
$$\varepsilon_\ell(C_t^{B;t_0}) = \varepsilon_k(C_{t_0}^{-B_t;t}) = \varepsilon \ .$$
Using the two equalities in \eqref{eq:C'-thru-C-left-end-d}, we conclude that
the $\ell$th column of~$C_{t}^{B_1;t_1}$ is $-\varepsilon \ee_k$, while the
$k$th column of~$C_{t_0}^{-B_{t'};t'}$ is $- \varepsilon \ee_\ell$, hence
$$\varepsilon_\ell(C_{t}^{B_1;t_1}) = \varepsilon_k(C_{t_0}^{-B_{t'};t'}) = - \varepsilon \ .$$

Combining \eqref{eq:C'-thru-C-right-end-two-times} and \eqref{eq:C'-thru-C-left-end-d}, we can rewrite the desired
equality \eqref{eq:C'-thru-C-left-end-d+1} as follows:
$$
(J_k + [-\varepsilon B]_+^{k \bullet}) C_t^{B;t_0} (J_\ell + [- \varepsilon B_t]_+^{\ell \bullet})
=
(J_k + [\varepsilon B]_+^{k \bullet})C_{t}^{B;t_0} (J_\ell + [\varepsilon B_t]_+^{\ell \bullet})
 \ ,$$
or equivalently (in view of \eqref{eq:J+B-ell-transpose-inverse}) as
$$(J_k + [\varepsilon B]_+^{k \bullet})
(J_k + [-\varepsilon B]_+^{k \bullet})C_{t}^{B;t_0} =
 C_t^{B;t_0} (J_\ell + [\varepsilon B_t]_+^{\ell \bullet})(J_\ell + [-\varepsilon B_t]_+^{\ell \bullet}) \ .$$
Performing the matrix multiplication and remembering \eqref{eq:b+identity}, we can simplify the last equality to
$$(I + \varepsilon B^{k \bullet}) C_{t}^{B;t_0} =
 C_t^{B;t_0} (I + \varepsilon B_t^{\ell \bullet}) \ ,$$
 or even further to
\begin{equation}
\label{eq:BC=CBt}
B^{k \bullet} C_{t}^{B;t_0} =  C_t^{B;t_0} B_t^{\ell \bullet} \ .
\end{equation}

To prove \eqref{eq:BC=CBt} we recall the identity \eqref{eq:G-C-inverse}, and in particular the matrix
$G_{t}^{B;t_0}$ that appears there.
To simplify the notation, we abbreviate
$$G_{t}^{B;t_0} = G, \quad C_{t}^{B;t_0} = C \ ;$$
as shown in the proof of \eqref{eq:G-C-inverse} given above, it
can be rewritten as follows:
\begin{align}
\label{eq:GT1}
G^T = (C_t^{-B^T;t_0})^{-1} = D C^{-1} D^{-1} \ ,
\end{align}
where $D$ is the diagonal matrix with positive diagonal entries such that
$- B^T = D B D^{-1}$.
In particular, the $k$th column of~$G^T$ is equal to $d_\ell d_k^{-1} \varepsilon \ee_\ell$.
However, we know that both $G^T$ and its inverse $C_t^{-B^T;t_0}$ are integer matrices.
Therefore, $d_\ell = d_k$, and we have
\begin{equation}
\label{eq:G-kth-row}
\text{the $k$th row of $G$ is equal to $\varepsilon \ee_\ell^T$.}
\end{equation}
Now we recall the  equality \eqref{eq:GBt=BC}.
Computing the entries in the $k$th row on both sides of \eqref{eq:GBt=BC} and using \eqref{eq:G-kth-row},
we get, for any $j = 1, \dots, n$:
$$\varepsilon b_{\ell j; t} = \sum_p b_{kp} c_{pj} \ .$$
Rewriting this in the matrix form yields \eqref{eq:BC=CBt}, thus completing the proofs of
Proposition~\ref{pr:C-mutation-left-end} and identity \eqref{eq:C-C-opposite}.
\end{proof}

\section{Some corollaries}
\label{sec:corollaries}

In this section we show that the assumption \eqref{eq:C-sign-coherence} and the identities \eqref{eq:G-C-inverse} and \eqref{eq:C-C-opposite}
imply most of the conjectures made in \cite{ca4} and recast in \cite{dwz2}, namely
Conjectures~1.1 - 1.4 and 1.6 from \cite{dwz2}.
For the convenience of the reader we reproduce their statements.

\begin{conjecture}
\label{con:F-CT-1-2}
\begin{enumerate}
\item[(i)]
Each polynomial $F_{j;t}^{B;t_0}$ has constant term~$1$.
\item[(ii)]
Each polynomial $F_{j;t}^{B;t_0}$ has a unique
monomial of maximal degree.
Furthermore, this monomial has coefficient~$1$, and it is
divisible by all the other occurring monomials.
\item[(iii)]
For every~$t \in \TT_n$,
the vectors $\gg_{1;t}^{B;t_0}, \dots, \gg_{n;t}^{B;t_0}$ are
sign-coherent, i.e., for any $i = 1, \dots, n$, the $i$-th components of all
these vectors are either all nonnegative, or all nonpositive.
\item[(iv)]
For every~$t \in \TT_n$, the vectors
$\gg_{1;t}^{B;t_0}, \dots, \gg_{n;t}^{B;t_0}$
form a $\ZZ$-basis of the lattice~$\ZZ^n$.
\item[(v)]
Let $t_0 \overunder{k}{} t_1$ be two adjacent vertices in~$\TT_n$,
and let $B' = \mu_k(B)$.
Then, for any $t \in \TT_n$ and $j = 1, \dots, n$,
the $\gg$-vectors $\gg_{j;t}^{B;t_0} = (g_1, \dots, g_n)$
and $\gg_{j;t}^{B';t_1} = (g'_1, \dots, g'_n)$ are related as
follows:
\begin{equation}
\label{eq:Langlands-dual-trop}
g'_i =
\begin{cases}
-g_k  & \text{if $i = k$};\\[.05in]
g_i + [b_{ik}]_+ g_k
  - b_{ik} \min(g_k,0)
 & \text{if $i \neq k$}.
\end{cases}
\end{equation}
\end{enumerate}
\end{conjecture}

\begin{proposition}
\label{pr:c-coherence-implies-conjectures}
The assumption \eqref{eq:C-sign-coherence} implies all the statements in Conjecture~\ref{con:F-CT-1-2}.
\end{proposition}

\begin{proof}
The equivalence of (i) and (ii) was shown in \cite[Section~5]{ca4} (see Conjectures~5.4, 5.5 and Proposition~5.3 there).
The fact that \eqref{eq:C-sign-coherence} implies conjectures (i) and (ii) was shown in \cite[Proposition~5.6]{ca4}
(note that \eqref{eq:C-sign-coherence} appears as condition (ii$'$) in the proof).

\smallskip

Part (iii) is immediate from \eqref{eq:C-sign-coherence} and \eqref{eq:G-C-opposite}.

\smallskip

Part (iv) can be restated by saying that the matrix
$G_t^{B;t_0}$ with columns $\gg_{j;t}^{B;t_0}$ is invertible over~$\ZZ$.
But this is immediate from \eqref{eq:G-C-inverse}.

\smallskip

It remains to prove (v).
Replacing $\min(g_k,0)$ with $- [- g_k]_+$, and using
\eqref{eq:c+b+identity}, we rewrite the desired equality \eqref{eq:Langlands-dual-trop} in the matrix form as follows:
\begin{equation}
\label{eq:Langlands-dual-trop-matrix}
G_{t}^{B';t_1} = (J_k + [\varepsilon B]_+^{\bullet k}) G_{t}^{B;t_0} + B
[- \varepsilon G_{t}^{B;t_0}]_+^{k \bullet} \ ,
\end{equation}
for any choice of sign $\varepsilon = \pm 1$.
Taking transpose matrices on both sides of \eqref{eq:Langlands-dual-trop-matrix} and using
\eqref{eq:G-C-opposite}, we can rewrite \eqref{eq:Langlands-dual-trop-matrix} as
\begin{equation}
\label{eq:Langlands-dual-trop-C-version}
C_{t_1}^{B_t^T;t} = C_{t_0}^{B_t^T;t} (J_k + [\varepsilon B^T]_+^{k \bullet})  +
[- \varepsilon C_{t_0}^{B_t^T;t}]_+^{\bullet k} B^T \ .
\end{equation}
In view of \eqref{eq:C-sign-coherence-matrix-form},
choosing $\varepsilon = \varepsilon_k(C_{t_0}^{B_t^T;t})$, we see that \eqref{eq:Langlands-dual-trop-C-version}
takes the form
\begin{equation}
\label{eq:Langlands-dual-trop-C-version-short}
C_{t_1}^{B_t^T;t} = C_{t_0}^{B_t^T;t} (J_k + [\varepsilon_k(C_{t_0}^{B_t^T;t})\ B^T]_+^{k \bullet}) \ ,
\end{equation}
which is just the first equality in \eqref{eq:C'G'-thru-CG-right-end} (with $(t,t')$ replaced by $(t_0, t_1)$, and
$(B;t_0)$ replaced by $(B_t^T;t)$).
This proves \eqref{eq:Langlands-dual-trop-matrix} and completes the proof of
Proposition~\ref{pr:c-coherence-implies-conjectures}.
\end{proof}

\end{document}